\documentclass[11pt]{amsart}
\usepackage{amssymb}
\usepackage{amsfonts}

\usepackage{bbm}
\usepackage{mathrsfs}
\usepackage{color}

\usepackage{array}

\usepackage{amsfonts,amsmath, amssymb}

\newcommand{\bea}{\begin{eqnarray}}

\newcommand{\eea}{\end{eqnarray}}

\newcommand{\be}{\begin {equation}}

\newcommand{\ee}{\end{equation}}

\newtheorem{theorem}{Theorem}[section]

\newtheorem{corollary}[theorem]{Corollary}
\newtheorem{lemma}[theorem]{Lemma}

\newtheorem{remark}[theorem]{Remark}

\newtheorem{claim}[theorem]{Claim}

\begin{document}

\title{quadratic Lagrange spectrum: I}

\author {Xianzu Lin }

\date{ }
\maketitle
   {\small \it College of Mathematics and Computer Science, Fujian Normal University, }\\
    \   {\small \it Fuzhou, {\rm 350108}, China;}\\
      \              {\small \it Email: linxianzu@126.com}
\begin{abstract}
In this paper we prove the existence of Hall's ray for the
quadratic Lagrange spectrums of all real quadratic numbers. For a
large class of real quadratic numbers, we compute the Hurwitz
constants of their quadratic Lagrange spectrums.

\end{abstract}






\section{Introduction}
By the classical Dirichlet theorem, for any irrational number
$\alpha$, there exist infinitely many integers $p,q$ such that
\begin{equation}\label{e90x} |\alpha-\tfrac{p}{q}|<\tfrac{1}{q^2}.\end{equation}
If we want to strengthen this ineqality by replacing $1$ with a
smaller constant, we need to consider
$$L(\alpha):=\liminf_{q\rightarrow\infty}q|q\alpha-p|.$$
The classical Lagrange spectrum is defined to be the set
$\emph{L}$ of values $L(\alpha)$ where $\alpha$ runs over all
irrational numbers. In 1879, Markoff \cite{m,m1} proved that
$\emph{L}$ begins with a discrete sequence: $\tfrac{1}{\sqrt{5}},$
$\tfrac{1}{\sqrt{8}}$, $\tfrac{\sqrt{221}}{5},\cdots$ which
converges to $\tfrac{1}{3}$, where $\tfrac{1}{\sqrt{5}}$ is the
Hurwitz constant. In 1947, Hall \cite{h} showed  that $\emph{L}$
contains a nontrivial interval $[0,\epsilon] (Hall's\  ray)$.  In
1975, Cusick \cite{c} proved that $\emph{L}$  is a closed subset
of $\mathbb{R}$. We refer the readers to \cite{cf} and the
references therein for more properties of Lagrange spectrum.

In \cite{pp,pp1}, as a corollary of their geometric
generalizations of the Lagrange Spectra in negative curvature,
Parkkonen and Paulin  defined the quadratic Lagrange spectrum as
follows

Throughout this paper, let $^{\sigma}$ be the Galois conjugate of
a quadratic number $x$. Let $\alpha$ be a fixed real quadratic
 number, and let
$\Theta_{\alpha}=PSL(2,\mathbb{Z})\{\alpha, \alpha^{\sigma}\}$ be
the orbit of $\alpha$ and $\alpha^{\sigma}$ for the action of
$PSL(2,\mathbb{Z})$. The quadratic Lagrange spectrum of $\alpha$
is defined to be the set $Sp_{\alpha}$ of values
$$c_{\alpha}\xi:=\liminf_{\beta\in\Theta_{\alpha}  \atop |\beta-\beta^{\sigma}|\rightarrow0}2\tfrac{|\xi-\beta|}{|\beta-\beta^{\sigma}|}$$ where $\xi$
runs over all real number not in $\mathbb{Q}\cup\Theta_{\alpha}$.

Parkkonen and Paulin  \cite{pp,pp1} showed that $Sp_{\alpha}$ is a
closed subset of $\mathbb{R}$. But their geometric method does not
imply the existence of Hall's ray for $Sp_{\alpha}$. The first
main result of this paper is existence of Hall's ray for
$Sp_{\alpha}$:
 \begin{theorem} \label{main1}
For any real quadratic number $\alpha$, there exists a positive
number  $\psi=\psi(\alpha)$ (see Section 3 for the definition of
$\psi(\alpha)$) such that $[0,\psi]\subset Sp_{\alpha}$.
\end{theorem}

Define the Hurwitz constant $K_{\alpha}$ of $\alpha$ to be the
maximum of $Sp_{\alpha}$, Parkkonen and Paulin  \cite{pp,pp1}
showed that $K_{\alpha}\leq (1+\sqrt{2})\sqrt{3}$ for each real
quadratic
 number $\alpha$.
 In \cite{bu}, Bugeaud pointed out that the theory of continued
 fractions is well suited for the investigation of quadratic Lagrange
spectrum. Among other results, he showed that $K_{\alpha}\leq
\tfrac{1}{2}$ and $K_{\varphi}=\tfrac{3}{\sqrt{5}}-1$, where
$\varphi$ is the Golden Ratio $(1+\sqrt{5})/2$. Bugeaud
conjectured that $\tfrac{3}{\sqrt{5}}-1$ is  a common upper bound
for all the  $K_{\alpha}$.  Pejkovi\'{c}\cite{pe}  proved this
conjecture using the theory of continued
 fractions. He also showed that $$\lim_{v\rightarrow\infty}K_{\beta_{u,v}}=\tfrac{3}{\sqrt{5}}-1-\tfrac{2}{u\sqrt{5}}$$
where $\beta_{u,v}=[\overline{u,v}]$ and $v\geq u>1/0.0008$.

In this paper, we determine the Hurwitz constant $K_{\beta_{u,v}}$
for $v\geq u\geq9$.

 \begin{theorem} \label{main3}
For $v\geq u\geq9$, the Hurwitz constant
$$K_{\beta_{u,v}}=\frac{2(\tau_1-\varphi^2)(1+\tfrac{\tau_2}{\varphi^2})}{(\varphi+\tfrac{1}{\varphi})(\tau_1
\tau_2+1)},$$ where $\tau_1=[\overline{u,v}]$ and
$\tau_2=[\overline{v,u}]$.
\end{theorem}
\begin{remark}\label{remar}
 In \cite{bu}, Bugeaud conjectured  the following formula for $K_{\alpha}$
 $$K_{\alpha}=\max\{c_{\alpha}\varphi, \lim_{m\rightarrow\infty} c_{\alpha}\varphi_m\},$$
 where $\varphi_m=[\overline{m}]$.
   We calculate both $c_{\alpha}\varphi$ and $\lim_{m\rightarrow\infty}
   c_{\alpha}\varphi_m$ for $\alpha=\beta_{u,v}$, and prove that the maximal of them is
   indeed $K_{\beta_{u,v}}$ when $v\geq u\geq9$.  In fact a more
   delicate method shows that Theorem \ref{main3} also holds when $v\geq u\geq4$.
\end{remark}
\begin{remark}\label{remar}
Our proof of  Theorem \ref{main3} also implies that
$K_{\beta_{u,v}}$ is an isolated point of $Sp_{\beta_{u,v}}$, when
$v\geq u\geq9$. It is interesting to see whether
$Sp_{\beta_{u,v}}$ begins with a discrete sequence, as in the case
of classical Lagrange spectrum.
\end{remark}

This paper is structured as follows: In Section 2, we give
preliminaries that will be used throughout this paper. In Section
3, we give proof  of Theorem \ref{main1}. In Section 4, we
determine the Hurwitz constant $K_{\beta_{u,v}}$ for $v\geq
u\geq9$. The author thanks professor Pejkovi\'{c} for his careful
reading of this paper, and numerous suggestions and corrections
for improvement.

 \section{preliminaries}
In this paper, we write $$[a_0,a_1,a_2,\cdots,a_n]$$ for the
finite continued fraction expansion
$$a_0+\frac{1}{a_1+\frac{1}{a_2+\cdots+\frac{1}{a_n}}},$$ and write $$[a_0,a_1,a_2,\cdots,a_n,\cdots]$$ for the
infinite continued fraction expansion
$$a_0+\frac{1}{a_1+\frac{1}{a_2+\cdots+\frac{1}{a_n+\cdots}}},$$
where $a_1,a_2,\cdots, $ are positive integers and $a_0$ is an
integer. An eventually periodic continued fraction is written as
$$[a_0,a_1,a_2,\cdots,a_{k-1},\overline{a_k,\cdots,a_{k+m-1}}],$$
where $a_0,a_1,a_2,\cdots,a_{k-1}$ is the preperiod and
$a_k,\cdots,a_{k+m-1}$ is the shortest period.

The sequence of convergents of
$$a=[a_0;a_1,a_2,\cdots,a_n,\cdots]$$ is defined by
$$p_{-2}=0, \   p_{-1}=1, \   p_{n}=a_np_{n-1}+p_{n-2}  \   (n\geq 0),$$
$$q_{-2}=1, \   q_{-1}=0, \   q_{n}=a_nq_{n-1}+q_{n-2}  \   (n\geq 0).$$

A direct proof  by induction shows that
\begin{equation}\label{e90}q_{m+n}\geq
2^{\tfrac{m-1}{2}}q_{n},\end{equation}
 for $m,n\geq 1.$

 \begin{lemma} (cf.\cite{hw})\label{ss}
$$[a_0;a_1,a_2,\cdots,a_n]=\frac{p_{n}}{q_{n}},$$
$$|a-\frac{p_{n}}{q_{n}}|<\frac{1}{q_{n}q_{n+1}}.$$

   \end{lemma}

  \begin{lemma}\label{ss091}
Let $$\alpha=[a_0,a_1,a_2,\cdots],$$ and
$$\beta=[b_0,b_1,b_2,\cdots]$$ be  the continued fraction expansions of two
real numbers, and let $\{\frac{p_{n}}{q_{n}}\}_{n\geq0}$ the
sequence of convergents of $\alpha$. Let $n$ be a nonnegative
integer such that $a_i=b_i$ for $i=1,\cdots,n-1$, and $a_{n}\neq
b_{n}$. Then we have
$$\frac{1}{72q_{n}q_{n+2}}\leq\frac{1}{72q^2_{n}a_{n+1}a_{n+2}}\leq|\alpha-\beta|
 .$$

Moreover, if there exists a positive integer $C$ such that
$|a_i|\leq C$, then we
 have $$\frac{1}{72(C+1)^{2n+2}}\leq|\alpha-\beta|.$$
   \end{lemma}
   \begin{proof}
The first inequality  follows directly from \cite[Lemma 2.2]{bu}
and its proof. The second follows from the first since $q_{n}\leq
(C+1)^{n}$.
\end{proof}

     \begin{lemma}\label{ss0912}
Let $$\alpha=[a_0,a_1,a_2,\cdots],$$
$$\beta=[b_0,b_1,b_2,\cdots],$$ and
$$\gamma=[c_0,c_1,c_2,\cdots]$$ be  the continued fraction expansions of
three real numbers. Let $m$ be a positive integer such that
$a_i=b_i$ for $i=1,\cdots,m-1$, and $a_{m}\neq b_{m}$, and let $n$
be a positive integer such that $a_i=c_i$ for $i=1,\cdots,n-1$,
and $a_{n}\neq c_{n}$. Then if $m>n+2$, we have
$$|\alpha-\beta|\leq\frac{72}{2^{m-n-3}}|\alpha-\gamma|.$$If $0\leq m-n\leq 2$, $n\geq2$ and $a_{m},a_{m+1},a_{m+2}\leq D$, we have
$$|\alpha-\beta|\leq 4\cdot72D^4|\alpha-\gamma|.$$
   \end{lemma}
\begin{proof}
Let $\{\frac{p_{i}}{q_{i}}\}_{i\geq0}$ be the sequence of
convergents of $\alpha$ and let
$\{\frac{p'_{i}}{q'_{i}}\}_{i\geq0}$ be the sequence of
convergents of $\beta$. Then by (\ref{e90}) and Lemmas \ref{ss}
and \ref{ss091}, if $m>n+2$,
$$|\alpha-\beta|\leq\max(\frac{1}{q_{m}q_{m-1}}, \frac{1}{q'_{m}q'_{m-1}}) \leq\frac{1}{2^{m-n-3}q_{n}q_{n+2}}  \leq\frac{72}{2^{m-n-3}} |\alpha-\gamma|
 .$$
Now we assume that $a_{m},a_{m+1},a_{m+2}\leq D$. If $m-n=1,2$, we
have
 $$|\alpha-\beta|\leq\max(\frac{1}{q_{m}q_{m-1}}, \frac{1}{q'_{m}q'_{m-1}})  \leq\frac{1}{q_{m-1}^{2}} \leq\frac{4D^2}{q_{n}q_{n+2}} \leq 4\cdot72D^2|\alpha-\gamma|
 .$$

 If  $m=n$, set $$\alpha'=[a_{m},a_{m+1},a_{m+2},\cdots],$$
$$\beta'=[b_{m},b_{m+1},b_{m+2},\cdots],$$ and
$$\gamma'=[c_{m},c_{m+1},c_{m+2},\cdots].$$ Then we have
$$\alpha=\frac{p_{m-1}\alpha'+p_{m-2}}{q_{m-1}\alpha'+q_{m-2}},$$
$$\beta=\frac{p_{m-1}\beta'+p_{m-2}}{q_{m-1}\beta'+q_{m-2}},$$
and
$$\gamma=\frac{p_{m-1}\gamma'+p_{m-2}}{q_{m-1}\gamma'+q_{m-2}}.$$
As the function $\frac{p_{m-1}x+p_{m-2}}{q_{m-1}x+q_{m-2}}$ is
monotone, we have
 {\setlength{\arraycolsep}{0pt}
\begin{eqnarray*}\label{for90}
&&|\alpha-\beta|\\
&\leq&\max(|\alpha-\frac{p_{m-1}}{q_{m-1}}|,|\alpha-\frac{p_{m-1}+p_{m-2}}{q_{m-1}+q_{m-2}}|)
\\
&\leq&\max(\frac{1}{q_{m-1}(q_{m-1}\alpha'+q_{m-2})},\frac{|\alpha'-1|}{(q_{m-1}\alpha'+q_{m-2})(q_{m-1}+q_{m-2})})\\
&\leq&\frac{D}{q_{m-1}(q_{m-1}\alpha'+q_{m-2})}.
\end{eqnarray*}
}

On the other hand, by Lemma \ref{ss091}, $$|\alpha-\gamma|=
\frac{|\alpha'-\gamma'|}{(q_{m-1}\alpha'+q_{m-2})(q_{m-1}\gamma'+q_{m-2})}\geq\frac{1}{4\cdot72D^3q_{m-1}(q_{m-1}\alpha'+q_{m-2})}.$$
Hence the second assertion of the lemma follows.
\end{proof}
The next lemma follows directly from an elementary calculation.
  \begin{lemma}\label{ss0913}
Let $$\cdots,a_{-1},a_0,a_1,a_2,\cdots,$$ and
$$\cdots,b_{-1},b_0,b_1,b_2,\cdots,$$ be  two doubly infinite
sequences of positive integer, satisfying $a_i=b_i$ for $m\leq
i\leq n$. Set
$$A_i=\frac{([a_{i+1},a_{i+2},\cdots]-[b_{i+1},b_{i+2},\cdots])(\tfrac{1}{[a_{i},a_{i-1}\cdots]}-\tfrac{1}{[b_{i},b_{i-1}\cdots]})}{([a_{i+1},a_{i+2},\cdots]+\tfrac{1}{[a_{i},a_{i-1}\cdots]})([b_{i+1},b_{i+2},\cdots]+\tfrac{1}{[b_{i},b_{i-1}\cdots]})}.$$
  Then we have $$A_{m-1}=A_{m}=\cdots=A_{n}.$$ \end{lemma}

 We recall that the continued fraction expansion of a real quadratic number $\alpha$ is eventually periodic. It is also well-known that if there exists an
  $A\in PSL(2,\mathbb{Z})$ such that $\beta=A\alpha$
, then the continued fraction expansions of $\alpha$ and $\beta$
have the same tail. Thus throughout this paper, we assume without
loss of generality that $\alpha=[\overline{b_1,\cdots,b_s}]$ where
$b_1,\cdots,b_s$ and $s$ are positive integers. We extend the
notation $b_i$ to all $i\in\mathbb{Z}$ by requiring $b_i=b_j$
whenever $i\equiv j\ (mod\ s)$.  It is well-known that the Galois
conjugate of $\alpha$ is
$$\alpha^{\sigma}=-\frac{1}{\alpha'}=-\frac{1}{[\overline{ b_s,\cdots,b_1}]}.$$

Let  $B=2+max(b_1,b_2,\cdots,b_s)$.

The following lemma is from \cite[Corollary 2.3]{bu}.
    \begin{lemma}\label{ssc}
Let $$\xi=[a_0,a_1,a_2,\cdots]$$ be an irrational real number. Let
$$\beta=[a_0,a_1,\cdots,a_r,a'_1,\cdots,a'_t, \overline{b_{i+1},\cdots,b_{i+s}}]$$  where $t\geq1$, $a_{r+1}\neq
a'_1$ and $a'_t\neq b_{i}$. Then we have
$$2\frac{|\xi-\beta|}{|\beta-\beta^{\sigma}|}\geq\frac{1}{96B^3(B+1)^2}.$$
   \end{lemma}

Set $$\alpha_i=[\overline{b_{i},  \cdots, b_{i+s-1}}],$$
$$\alpha'_i=[\overline{b_{i+s-1},  \cdots, b_{i}}]=-\tfrac{1}{\alpha_i^{\sigma}},$$ and set $$\Phi=\{\alpha_1,
\cdots,\alpha_s,\alpha'_1, \cdots,\alpha'_s\}.$$

For each $\phi\in\Phi$, set
$$\xi_{r,\phi}=[a_0,a_1,\cdots,a_r,\phi].$$ By Lemma \ref{ssc}, when $c_{\alpha}\xi<\tfrac{1}{48B^3(B+1)^2},$ we
have
\begin{equation}\label{e4}c_{\alpha}\xi=\min_{\phi\in\Phi}\liminf_{r\rightarrow\infty}2\tfrac{|\xi-\xi_{r,\phi}|}{|\xi_{r,\phi}-\xi_{r,\phi}^{\sigma}|}.\end{equation}

Now we repeat the calculations in (cf.\cite{bu,pe}) to get an
explicit expression for
$2\tfrac{|\xi-\xi_{r,\phi}|}{|\xi_{r,\phi}-\xi_{r,\phi}^{\sigma}|}.$

Let $\{\frac{p'_{n}}{q'_{n}}\}_{n\geq0}$ be the sequence of
convergents of $[a_0,a_1,a_2,\cdots]$. Then by an elementary
property of continued fractions (cf.\cite[p.133]{hw}) we have
$$\xi_{r,\phi}=\frac{p'_{r}\phi+p'_{r-1}}{q'_{r}\phi+q'_{r-1}}.$$
Hence
$$|\xi_{r,\phi}-\xi_{r,\phi}^{\sigma}|=|\frac{\phi-\phi^{\sigma}}{(q'_{r}\phi+q'_{r-1})(q'_{r}\phi^{\sigma}+q'_{r-1})}|.$$

 Set
$$d_{r}:=\frac{q'_{r}}{q'_{r-1}}=[a_r,a_{r-1},\cdots,a_1]$$ and
$$D_{r}:=[a_{r+1},a_{r+2},\cdots].$$
 Then {\setlength{\arraycolsep}{0pt}
\begin{eqnarray*}
|\xi-\xi_{r,\phi}|&=&|\frac{p'_{r}D_{r}+p'_{r-1}}{q'_{r}D_{r}+q'_{r-1}}-\frac{p'_{r}\phi+p'_{r-1}}{q'_{r}\phi+q'_{r-1}}|\\
&=&|\frac{\phi-D_{r}}{(q'_{r}\phi+q'_{r-1})(q'_{r}D_{r}+q'_{r-1})}|.\\
\end{eqnarray*}
} Hence we have
\begin{equation}\label{d4ss3}\frac{2|\xi-\xi_{r,\phi}|}{|\xi_{r,\phi}-\xi_{r,\phi}^{\sigma}|}=\frac{2|D_{r}-\phi||\tfrac{1}{d_{r}}+\phi^{\sigma}|}{|\phi-\phi^{\sigma}||D_{r}+\tfrac{1}{d_{r}}|}.\end{equation}

 \section{Hall's ray for quadratic Lagrange spectrum}
In this section, we prove the existence of Hall's ray for
quadratic Lagrange spectrum. Let $n$ be the minimal positive
integer satisfying
\begin{equation}\label{sxs71vb}2^{(n-1)s}>16\cdot16\cdot36\cdot144B^{2s+8}.\end{equation}
It follows from (\ref{sxs71vb}) that $n\geq3$. Set
$\psi=\frac{1}{72\cdot72B^{s(6n+6)+4}}$ and let
$\epsilon\in(0,\psi]$.

    \begin{lemma}\label{ss376}
\begin{equation}\label{ess376}\epsilon=\frac{2(B+\frac{1}{x}-\alpha')(y-\alpha)}{(\alpha\alpha'+1)(1+y(B+\frac{1}{x}))}\end{equation}
defines an increasing  function $f_{\epsilon}(x)=y$ which maps
$(0,+\infty)$ injectively  into $(\alpha-1/2,\alpha+1/2)$. On the
interval $[B,B+2]$, we have $M/2\leq f'_{\epsilon}(x)\leq 2M$
where
$$M=\frac{\epsilon(\alpha\alpha'+1)[2\alpha\alpha'+2-\epsilon(\alpha\alpha'+1)]}{[(2(B-\alpha')-\epsilon
B(\alpha\alpha'+1))(B+1)+2-\epsilon(\alpha\alpha'+1)]^2}.$$
   \end{lemma}
\begin{proof}
Solving the equation, we get
$$f_{\epsilon}(x)=y=\frac{2\alpha[(B-\alpha')x+1]+\epsilon(\alpha\alpha'+1)x}{2[(B-\alpha')x+1]-\epsilon
(\alpha\alpha'+1)(Bx+1)}.$$

A direct calculation shows that
$$f'_{\epsilon}(x)=\frac{\epsilon(\alpha\alpha'+1)[2\alpha\alpha'+2-\epsilon(\alpha\alpha'+1)]}{[(2(B-\alpha')-\epsilon
B(\alpha\alpha'+1))x+2-\epsilon(\alpha\alpha'+1)]^2}.$$ Hence the
lemma follows.
\end{proof}

 \begin{lemma}\label{ssxa}For each $\epsilon\in(0,\psi]$
There exist a positive integer $k>3n$ and  two positive infinite
continued fractions $$x=[e_1,e_2,\cdots]$$ and
$$y=[\underbrace{b_1,\cdots,b_s}_{k},f_1,f_2,\cdots]$$ such that
\begin{enumerate} \item
equation (\ref{ess376}) is valid, i.e., $f_{\epsilon}(x)=y$;
 \item
  $(f_1,f_2,\cdots,f_s)\neq (b_1,\cdots,b_s)$;
 \item
 any subblock of $[\overline{b_1\cdots b_s}]$ of length $> ns$
does not occurs in $e_1e_2\cdots$ and $f_{m+1}f_{m+2}\cdots$.
\end{enumerate}

   \end{lemma}

\begin{proof}
Let $\Gamma$ be the uncountable  set of infinte words of elements
from $\{B,B+1\}$. Choose a
$\omega_0=u^{(0)}_1u^{(0)}_2\cdots\in\Gamma$ and set
$\mu_0=[u^{(0)}_1,u^{(0)}_2,\cdots]$ such that
 $\nu_0=f_{\epsilon}(\mu_0)$ is irrational, i.e.,
$\nu_0$ has an infinite continued fraction expansion
$$\nu_0=[\underbrace{b_1,\cdots,b_s}_{k},v^{(0)}_1,v^{(0)}_2,\cdots]$$
where $k\geq0$ indicates the number of times the block
$b_1,\cdots,b_s$ is repeated. and
$$(v^{(0)}_1,v^{(0)}_2,\cdots, v^{(0)}_s)\neq (b_1,\cdots,b_s,).$$
Let $\{\frac{p_{n}}{q_{n}}\}_{n\geq1}$ be the sequence of
convergents of $[\overline{b_1,\cdots,b_s}]$. By Lemma \ref{ss091}
and (\ref{ess376}) we have
\begin{equation}\label{sxs71cd}\frac{1}{72B^{2s(k+3)}}\leq\frac{1}{72q^2_{s(k+3)}}\leq|\nu_0-\alpha|<\epsilon B^4\leq\frac{1}{72B^{s(6n+6)}},\end{equation} hence $k>3n$.

 \begin{claim}\label{5645}
For any $\omega_1=v^{(1)}_1 v^{(1)}_2 \cdots \in\Gamma$,
$$\nu_1=[\underbrace{b_1,\cdots,b_s}_{k},v^{(0)}_1,v^{(0)}_2,\cdots,v^{(0)}_{ns},v^{(1)}_1,v^{(1)}_2,\cdots]$$
lies in the interval $f_{\epsilon}([u^{(0)}_1,u^{(0)}_1+1])$.
\end{claim}
\begin{proof}The distance of $[u^{(0)}_1,u^{(0)}_2,\cdots]$ from the boundary of $[u^{(0)}_1,u^{(0)}_1+1]$ is at
least $\frac{1}{B+2}$. If the claim is invalid, then by Lemma
\ref{ss376} we have
$$|\nu_1-\nu_0|>\frac{M}{2(B+2)}>\frac{\epsilon}{32B^5}.$$ On the
other hand, by Lemma \ref{ss0912}, we have
$$|\nu_1-\nu_0|\leq\frac{72}{2^{(n-1)s-3}}
|\alpha-\nu_0|\leq\frac{72\epsilon B^4}{2^{(n-1)s-3}}.$$ Hence
$$\label{sxs71v}2^{(n-1)s}<8\cdot32\cdot72B^9.$$ This contradicts (\ref{sxs71vb}).
\end{proof}

Now, pick
$\nu_1=[\underbrace{b_1,\cdots,b_s}_{k},v^{(0)}_1,v^{(0)}_2,\cdots,v^{(0)}_{ns},v^{(1)}_1,v^{(1)}_2,\cdots]$
for  some $\omega_1=v^{(1)}_1 v^{(1)}_2 \cdots \in\Gamma$ such
that $v^{(0)}_{ns+1}\neq v^{(1)}_1$ and
$$f^{-1}_{\epsilon}(\nu_1)=\mu_1=[u^{(0)}_1,\cdots,u^{(0)}_{s_0},u^{(1)}_1,u^{(1)}_2,\cdots]$$
is irrational, where $u^{(0)}_{s_0+1}\neq u^{(1)}_1$. By Claim
\ref{5645}, $s_0>0$.

Next pick a
$$\omega_2=u^{(2)}_1,u^{(2)}_2,\cdots\in\Gamma$$ and set
$$\mu_2=[u^{(0)}_1,\cdots,u^{(0)}_{s_0},u^{(1)}_1,u^{(1)}_2,\cdots,u^{(1)}_{ns},u^{(2)}_1,u^{(2)}_2,\cdots]$$
such that $u^{(1)}_{ns+1}\neq u^{(2)}_1$ and $f_{\epsilon}(\mu_2)$
is irrational. Set $\nu_2=f_{\epsilon}(\mu_2)$. Then by Lemmas
\ref{ss0912} and \ref{ss376},
\begin{equation}\label{sxs71}|\nu_2-\nu_1|\leq 2M|\mu_2-\mu_1|\leq
\frac{72}{2^{(n-1)s-4}}M|\mu_0-\mu_1|\leq
\frac{144}{2^{(n-1)s-4}}|\nu_0-\nu_1|.\end{equation}

 \begin{claim}\label{5645}
There exists a positive $s_1$ such that
$$\nu_2=[\underbrace{b_1,\cdots,b_s}_{k},v^{(0)}_1,v^{(0)}_2,\cdots,v^{(0)}_{ns},v^{(1)}_1,\cdots,v^{(1)}_{s_1},v^{(2)}_1,v^{(2)}_2,\cdots]$$
where $v^{(1)}_{s_1+1}\neq v^{(2)}_1$.
\end{claim}
\begin{proof}If the continued fraction expansions of $f_{\epsilon}(\mu_2)$ and $\nu_1$ differ
before $v^{(1)}_2$, by Lemma \ref{ss0912} we would have
\begin{equation}|\nu_0-\nu_1|\leq4\cdot72(2B)^4|\nu_2-\nu_1|.\end{equation} This
contradicts (\ref{sxs71}).
\end{proof}

Again, pick a
$$\omega_3=v^{(3)}_1,v^{(3)}_2,\cdots\in\Gamma$$ and set
$$\nu_3=[\underbrace{b_1,\cdots,b_s}_{k},v^{(0)}_1,v^{(0)}_2,\cdots,v^{(0)}_{ns},v^{(1)}_1,\cdots,v^{(1)}_{s_1},v^{(2)}_1,\cdots,v^{(2)}_{ns},v^{(3)}_1,v^{(3)}_2,\cdots]$$
such that $v^{(2)}_{ns+1}\neq v^{(3)}_1$ and
$f^{-1}_{\epsilon}(\nu_3)$ is irrational. Continuing  in this way,
we finally get two infinite continued fractions
$$x=[U^{(0)},U^{(1)},U^{(2)},\cdots]$$and
$$y=[\underbrace{b_1,\cdots,b_s}_{k},V^{(0)},V^{(1)},V^{(2)},\cdots],$$
where $V^{(2l)}$ and $U^{(2l+1)}$ are words of length $ns$, and
$V^{(2l+1)}$ and $U^{(2l)}$ are finite words of elements from
$\{B,B+1\}$. Hence $x$ and $y$ satisfy Conditions (i), (ii) and
(iii).
\end{proof}

Now for each $\epsilon\in(0,\psi]$, choose two infinite continued
fractions $$x=[e_1,e_2,\cdots]$$ and
$$y=[\underbrace{b_1,\cdots,b_s}_{k},f_1,f_2,\cdots]$$ satisfying
Conditions of Lemma \ref{ssxa}. Set {\setlength{\arraycolsep}{0pt}
\begin{eqnarray*}
\xi&=&[h_1,h_2,\cdots]\\
&=&[e_{n},\cdots,e_1,B,\underbrace{b_1,\cdots,b_s}_{k},f_1,\cdots,
f_{n},\\
&&e_{2n},\cdots,e_1,B,\underbrace{b_1,\cdots,b_s}_{k},f_1,\cdots
f_{2n},\\
&&e_{3n},\cdots,e_1,B,\underbrace{b_1,\cdots,b_s}_{k},f_1,\cdots f_{3n},\\
&&\cdots\cdots\cdots].\\
\end{eqnarray*}
}

By Condition (iii) of Lemma \ref{ssxa}, any subblock of
$[\overline{b_1\cdots b_s}]$ of length $> 2ns$ does not occur in
$h_1,h_2,\cdots$ except for $\underbrace{b_1,\cdots,b_s}_{k}$. Let
$I$ be the set of all $i$ satisfying
$h_{i+1},\cdots,h_{i+ks}=\underbrace{b_1,\cdots,b_s}_{k}$. For
each $\phi\in\Phi$, set
$$\xi_{r,\phi}=[h_1,h_2,\cdots,h_r,\phi],$$

$$d_{r}=[h_r,h_{r-1},\cdots,h_1],$$ and
$$D_{r}=[h_{r+1},h_{r+2},\cdots].$$ Then by (\ref{d4ss3}) and Lemma \ref{ssxa}, we
have \begin{equation}\label{sxs1}\lim_{r\in I,\
r\rightarrow\infty}\frac{2|\xi-\xi_{r,\alpha}|}{|\xi_{r,\alpha}-\xi_{r,\alpha}^{\sigma}|}=\epsilon,\end{equation}
hence $c_{\alpha}\xi\leq\epsilon.$ By (\ref{e4}) and
(\ref{d4ss3}), in order to prove $c_{\alpha}\xi=\epsilon,$ we need
to estimate
$\tfrac{2|D_{r}-\phi||\tfrac{1}{d_{r}}+\phi^{\sigma}|}{|\phi-\phi^{\sigma}||D_{r}+\tfrac{1}{d_{r}}|}.$
Assume that $\phi=\alpha_{i}$. Then let $m,l$ be the nonnegative
integers such that
$$h_{r-m+1}=b_{i-m},\cdots,h_{r+1}=b_{i},\cdots,h_{r+l}=b_{i+l-1},$$
and
$$h_{r-m}\neq b_{i-m-1},\  \   h_{r+l+1}\neq b_{i+l}.$$ Now we
divide the estimation into 3 cases
   \begin{enumerate}
\item [(i)] $m+l\leq 3s$

By Lemma \ref{ss091} we have $|D_{r}-\alpha_{i}|\geq
\tfrac{1}{72B^{6s+2}}$ and
$|\tfrac{1}{d_{r}}-\tfrac{1}{\alpha'_{i}}|\geq
\tfrac{1}{72B^{6s+4}}$. If $h_{r+1}\leq B$, we have
$$\frac{2|D_{r}-\alpha_{i}||\tfrac{1}{d_{r}}-\tfrac{1}{\alpha'_{i}}|}{|\alpha_{i}+\tfrac{1}{\alpha'_{i}}||D_{r}+\tfrac{1}{d_{r}}|}\geq\frac{1}{2\cdot72\cdot72B^{12s+8}}>\epsilon.$$
If $h_{r+1}> B$, we have
$$\frac{2|D_{r}-\alpha_{i}|}{|D_{r}+\tfrac{1}{d_{r}}|}\geq\frac{1}{B}.$$
Hence
$$\frac{2|D_{r}-\alpha_{i}||\tfrac{1}{d_{r}}-\tfrac{1}{\alpha'_{i}}|}{|\alpha_{i}+\tfrac{1}{\alpha'_{i}}||D_{r}+\tfrac{1}{d_{r}}|}\geq\frac{1}{2\cdot72B^{6s+6}}>\epsilon.$$
 \item
[(ii)]$3s<m+l\leq 2ns$. In this case let
$$h_{r-m+1},\cdots,h_{r+l}=b_{h},\cdots,b_{s},\underbrace{b_{1},\cdots,b_{s}}_{g},b_{1},\cdots,b_{j}$$
where $1\leq h,j\leq s$. By
 Lemma \ref{ss0913}, we have
  \begin{equation}\label{sxs74v}\tfrac{2|D_{r}-\alpha_{i}||\tfrac{1}{d_{r}}-\tfrac{1}{\alpha'_{i}}|}{|\alpha_{i}+\tfrac{1}{\alpha'_{i}}||D_{r}+\tfrac{1}{d_{r}}|}=\frac{2|D_{r-m+s-h+1}-\alpha||d_{r-m+s-h+1}-\alpha'|}{(\alpha\alpha'+1)(D_{r-m+s-h+1}d_{r-m+s-h+1}+1)}.\end{equation}
By Lemma \ref{ss091}, we have
\begin{equation}\label{sxs72v}|d_{r-m+s-h+1}-\alpha'|\geq\frac{1}{72B^{2s+4}}.\end{equation}
By Lemma \ref{ss0912} and (\ref{ess376}), we have
\begin{equation}\label{sxs73v}|D_{r-m+s-h+1}-\alpha|\geq\frac{2^{ns-3}}{72}|\nu_0-\alpha|\geq\frac{2^{ns-3}}{72}\frac{2\epsilon}{3}=\frac{2^{ns-3}\epsilon}{3\cdot36}.\end{equation}
Now combining  (\ref{sxs74v}), (\ref{sxs72v}) and (\ref{sxs73v})
yields that
$$\tfrac{2|D_{r}-\alpha_{i}||\tfrac{1}{d_{r}}-\tfrac{1}{\alpha'_{i}}|}{|\alpha_{i}+\tfrac{1}{\alpha'_{i}}||D_{r}+\tfrac{1}{d_{r}}|}\geq\frac{2^{ns}\epsilon}{24\cdot36\cdot36B^{2s+8}}\geq\epsilon.$$
 \item [(iii)]$m+l>2ns$. In this case, we have $r-m\in I$ and, by
 Lemma \ref{ss0913}, $$\tfrac{2|D_{r}-\alpha_{i}||\tfrac{1}{d_{r}}-\tfrac{1}{\alpha'_{i}}|}{|\alpha_{i}+\tfrac{1}{\alpha'_{i}}||D_{r}+\tfrac{1}{d_{r}}|}
 =\frac{2|\xi-\xi_{r-m,\alpha}|}{|\xi_{r-m,\alpha}-\xi_{r-m,\alpha}^{\sigma}|}.$$
\end{enumerate}

The estimation in the  case $\phi=\alpha'_{i}$ can be dealt with
in a similar way. Hence we show that
$$\min_{\phi\in\Phi}\liminf_{r\rightarrow\infty}2\tfrac{|\xi-\xi_{r,\phi}|}{|\xi_{r,\phi}-\xi_{r,\phi}^{\sigma}|}=\epsilon.$$
 This finishes the proof of Theorem
\ref{main1}.

 \section{Hurwitz constant of quadratic Lagrange spectrum}
From now on we assume that $v\geq u\geq1$. In this section, we
compute Hurwitz constants of quadratic Lagrange spectrums for real
quadratic numbers $[\overline{u,v}]$, $u\geq9$.

Set $\tau_1=[\overline{u,v}]$ and $\tau_2=[\overline{v,u}]$.
Recall that $\varphi$ is the Golden Ratio
$$(1+\sqrt{5})/2=[1,1,1,\cdots].$$Set
\begin{equation}\theta_{u,v}=\begin{cases} (\tau_1-\varphi^2)(1+\tfrac{\tau_2}{\varphi^2})& \text{if}\ u\geq3 \\(\tau_1-\varphi)(\tfrac{\tau_2}{\varphi}-1)&\text{if}\  u=2 \\(\varphi-\tau_1)(\tfrac{\tau_2}{\varphi}-1) &\text{if}\ u=1 \end{cases}\end{equation}
A direct calculation shows
that\begin{equation}\label{77}\theta_{u,v}<\tfrac{\tau_1\tau_2}{\varphi^2}.\end{equation}

    \begin{lemma}\label{3ss376} If $(u,v)\neq(1,1)$,
$$c_{\tau_1}\varphi=\frac{2\theta_{u,v}}{(\varphi+\tfrac{1}{\varphi}))(\tau_1\tau_2+1)}.$$
   \end{lemma}
\begin{proof}
Set $ \varphi'=[\underbrace{1,\cdots,1}_{r}],$ and let
$\tau=\tau_1$ or $\tau_2$.  If we approximate $\varphi$ by
\begin{equation}\label{sxs1cv}\alpha=[\underbrace{1,\cdots,1}_{r},a_1,\cdots,a_n,\tau],\end{equation}
where $a_1\neq1$, we need to estimate
$\frac{2|\varphi-\alpha|}{|\alpha-\alpha^{\sigma}|}.$  Let
$\{\frac{p_{n}}{q_{n}}\}_{n\geq1}$ be the sequence of convergents
of $[a_1,\cdots,a_n]$. Applying the deduction of equation
(\ref{d4ss3}) in \S 2 shows that
$$\frac{2|\varphi-\alpha|}{|\alpha-\alpha^{\sigma}|}=
\frac{2|\tfrac{p_{n}\tau+p_{n-1}}{q_{n}\tau+q_{n-1}}-\varphi||\tfrac{p_{n}\tau^{\sigma}+p_{n-1}}{q_{n}\tau^{\sigma}+q_{n-1}}+\tfrac{1}{\varphi'}|}
{|\tfrac{p_{n}\tau+p_{n-1}}{q_{n}\tau+q_{n-1}}-\tfrac{p_{n}\tau^{\sigma}+p_{n-1}}{q_{n}\tau^{\sigma}+q_{n-1}}|(\varphi+\tfrac{1}{\varphi'})},$$
where we can assume that $\varphi'$ is arbitrarily closed to
$\varphi$ if necessary. Let $\tau'=-\tfrac{1}{\tau^{\sigma}}$.
Then the right hand side of the above equality simplifies to $$
\frac{2|(p_{n}-\varphi q_{n})\tau+(p_{n-1}-\varphi
q_{n-1})||(p_{n-1}+\varphi'^{-1}
q_{n-1})\tau'-(p_{n}+\varphi'^{-1} q_{n})|}
{(\tau\tau'+1)(\varphi+\tfrac{1}{\varphi'})}.$$ Now we follow the
arguments in \cite{pe} to treat the case $n=1$ (the case $n=0$ can
be reduced to the case $n=1$ by setting $a_1=[\tau]$ and replacing
$\tau$ with $\tau'$). In this case, $q_{n}=p_{n-1}=1$ and
$q_{n-1}=0$. We need to estimate
$$Q= \frac{2|(p_{1}-\varphi )\tau+1||\tau'-p_{1}-\varphi'^{-1} |}
{(\tau\tau'+1)(\varphi+\tfrac{1}{\varphi'})},$$ which is an
absolute value of a quadratic form in $p_{1}\in \mathbb{Z}$. The
minimal value can only be attained for integers closest to the
zeroes of the quadratic form, which are $\varphi-\tau^{-1}$ and
$\tau'-\varphi'^{-1}$. Hence the possible minimal integer points
are $$\begin{cases} 1,2,v-1,v & \text{if}\ u\geq2,  \tau=\tau_1
\\
1,2,u-1,u & \text{if}\ u\geq2, \tau=\tau_2
\\0, 1,v,v+1 &\text{if}\ u=1, \tau=\tau_1
\\0,1,2 &\text{if}\ u=1,
\tau=\tau_2
\end{cases}$$
Evaluating $Q$ on these integers we get
$$\liminf_{|\alpha-\alpha^{\sigma}|\rightarrow0}\frac{2|\varphi-\alpha|}{|\alpha-\alpha^{\sigma}|}=\frac{2\theta_{u,v}}{(\varphi+\tfrac{1}{\varphi})(\tau_1\tau_2+1)},$$
where we require $n=1$ and $\tau\in\{\tau_1,\tau_2\}$ in
(\ref{sxs1cv}).

Now we treat the case $n\geq2$ and $a_n\neq [\tau']$. In this
case, as $a_1\neq1$, we have $p_{i}\geq 2q_{i},$ and hence
\begin{equation}\label{sxs1cvce}p_{i}-\varphi q_{i}\geq p_{i}(1-\varphi q_{i}/p_{i})\geq 2-\varphi,\end{equation} for $i=n,n-1$.
(\ref{sxs1cvce}) still holds for $i=n-2$ because when $n=2$,
$p_{n-2}-\varphi q_{n-2}=1$. Thus it suffices to show that
$$\frac{2\tau(p_{n}-\varphi q_{n})|(p_{n-1}+\varphi'^{-1}
q_{n-1})\tau'-(p_{n}+\varphi'^{-1} q_{n})|}
{(\tau\tau'+1)(\varphi+\tfrac{1}{\varphi'})}\geq
\frac{2\theta_{u,v}}{(\varphi+\tfrac{1}{\varphi})(\tau_1\tau_2+1)},$$
or
$$\tau(p_{n}-\varphi
q_{n})|(p_{n-1}+\varphi'^{-1} q_{n-1})\tau'-(p_{n}+\varphi'^{-1}
q_{n})|\geq
\frac{(\varphi+\tfrac{1}{\varphi'})}{(\varphi+\tfrac{1}{\varphi})}\theta_{u,v}.$$

By (\ref{77}), when $\varphi'$ is sufficiently closed to
$\varphi$, we have
$$\frac{(\varphi+\tfrac{1}{\varphi'})}{(\varphi+\tfrac{1}{\varphi})}\theta_{u,v}
<\frac{\tau_1\tau_2}{\varphi^2}.$$

Hence it suffices to show that

\begin{equation}\label{3255}\frac{(p_{n}-\varphi
q_{n})|(p_{n-1}+\varphi'^{-1}
q_{n-1})(\tau'-a_n)-(p_{n-2}+\varphi'^{-1}
q_{n-2})|}{\tau'}\geq\frac{1}{\varphi^2}.\end{equation}

When $a_n> [\tau']$, the left hand side of (\ref{3255}) is larger
than
$$\frac{a_n(p_{n-1}-\varphi
q_{n-1})(p_{n-2}+\varphi'^{-1}
q_{n-2})}{\tau'}\geq(2-\varphi)=\frac{1}{\varphi^2}.$$

When $a_n< [\tau']$, the left hand side of (\ref{3255}) is
{\setlength{\arraycolsep}{0pt}
\begin{eqnarray*}\label{for90}
&& \frac{(p_{n}-\varphi q_{n})((p_{n-1}+\varphi'^{-1}
q_{n-1})(\tau'-a_n)-(p_{n-2}+\varphi'^{-1}
q_{n-2}))}{\tau'}\\
&\geq&\frac{(p_{n}-\varphi q_{n})((p_{n-1}+\varphi'^{-1}
q_{n-1})-(p_{n-2}+\varphi'^{-1}
q_{n-2}))}{a_n+1}\\
&\geq&\frac{(a_n(p_{n-1}-\varphi q_{n-1})+(p_{n-2}-\varphi q_{n-2}))}{a_n+1}\\
&\geq&(2-\varphi)=\frac{1}{\varphi^2}.
\end{eqnarray*}
} This finishes the proof of the lemma.
\end{proof}

For an irrational real number $$\xi=[a_0,a_1,a_2,\cdots],$$ set
$$N_{\xi}=\liminf_{r\rightarrow\infty} a_r,$$ and set $$\vartheta_{u,v}=\begin{cases}
1& \text{if}\ u\geq2  \\
 \\ \tau_1-1 &\text{if}\ u=1.
 \end{cases}$$
    \begin{lemma}\label{3sfs376} For any $v\geq u\geq1$, we have
$$\lim_{N_{\xi}\rightarrow\infty}
c_{\tau_1}\xi=\frac{2\vartheta_{u,v}}{(\tau_1\tau_2+1)}.$$ In
particular, if $N_{\xi}= \infty$, we have
$$c_{\tau_1}\xi=\frac{2\vartheta_{u,v}}{(\tau_1\tau_2+1)}.$$
   \end{lemma}
\begin{proof}Without loss of generality, we can assume that $$\xi=[b_0,b_1,b_2,\cdots]$$
is a positive number such that $N_{\xi}$ is very large. Set
$$d_{r}:=[b_r,b_{r-1},\cdots,b_0]$$ and
$$D_{r}:=[b_{r+1},b_{r+2},\cdots].$$ If we approximate $\xi$ by \begin{equation}\label{s}\alpha=[b_0,b_1,\cdots,b_r,a_1,\cdots,a_n,\tau],\end{equation}
where $a_1\neq b_{r+1}$ and $\tau\in\{\tau_1,\tau_2\}$, we need to
estimate
\begin{equation}\label{stas} \tfrac{2|(p_{n}-D_{r}
q_{n})\tau+(p_{n-1}-D_{r} q_{n-1})||(p_{n-1}+d_{r}^{-1}
q_{n-1})\tau'-(p_{n}+d_{r}^{-1} q_{n})|}
{(\tau\tau'+1)(D_{r}+\tfrac{1}{d_{r}})}\end{equation} or
\begin{equation}\label{st} \tfrac{2|(p_{n}-D_{r}
q_{n})\tau+(p_{n-1}-D_{r}
q_{n-1})||(p_{n-1}(\tau'-a_n)-p_{n-2})+d_{r}^{-1}(q_{n-1}(\tau'-a_n)-q_{n-2})|}
{(\tau\tau'+1)(D_{r}+\tfrac{1}{d_{r}})}\end{equation} where
$\{\frac{p_{i}}{q_{i}}\}_{i\geq1}$ is the sequence of convergents
of $[a_1,\cdots,a_n]$. The treatment in the case $n=1$ proceeds
exactly as above and implies the lower limit is
$\frac{2\vartheta_{u,v}}{(\tau_1\tau_2+1)}$ when
$N_{\xi}\rightarrow\infty$.

\bigskip

Now we treat the case $n\geq2$ and $a_n\neq [\tau']$. When $a_n>
[\tau']$, $$|p_{n-1}(\tau'-a_n)-p_{n-2}|\geq|
q_{n-1}(\tau'-a_n)-q_{n-2}|,$$ hence
\begin{equation}\label{st1cd} \frac{|p_{n-1}(\tau'-a_n)-p_{n-2}|}{|
q_{n-1}(\tau'-a_n)-q_{n-2}|}\geq
\frac{1}{\tau\tau'}.\end{equation} When $a_n< [\tau']$,
$$|p_{n-1}(\tau'-a_n)-p_{n-2}|\geq \frac{p_{n-1}}{\tau},$$ and
$$|
q_{n-1}(\tau'-a_n)-q_{n-2}|\leq p_{n-1}\tau',$$ hence
(\ref{st1cd}) still holds.

 As we are only concerned with the
the lower limit when $N_{\xi}\rightarrow\infty$, by (\ref{st1cd}),
we can replace (\ref{st}) with

\begin{equation}\label{st1} \frac{2|(p_{n}-D_{r}
q_{n})\tau+(p_{n-1}-D_{r} q_{n-1})||p_{n-1}(\tau'-a_n)-p_{n-2}|}
{D_{r}(\tau\tau'+1)}.\end{equation}

 \begin{claim}\label{5645vv}
$$|p_{n-1}(\tau'-a_n)-p_{n-2}|\geq \frac{p_{n-1}}{\tau} min(1,\tau-1).$$
\end{claim}
\begin{proof}
When $a_n> [\tau']$, the left hand side is
$$p_{n-1}(a_n-\tau')+p_{n-2}> p_{n-1}(1-\tfrac{1}{\tau}) .$$

When $a_n< [\tau']$, the left hand side is
$$p_{n-1}(\tau'-a_n)-p_{n-2}\geq p_{n-1}(1+\tfrac{1}{\tau})
-p_{n-2}\geq \frac{p_{n-1}}{\tau}.$$
\end{proof}

It remains to estimate $|(p_{n}-D_{r} q_{n})\tau+(p_{n-1}-D_{r}
q_{n-1})|.$ We note that as $a_1\neq b_{r+1}$, $(p_{n}-D_{r}
q_{n})\tau$ and $p_{n-1}-D_{r} q_{n-1}$ can not have opposite
signs.

 \begin{claim}\label{5645}
If $n=2$ and $a_2\neq1$, or $n=2$ and $a_1>b_{r+1}$, or $n>2$, we
have
$$|(p_{n}-D_{r} q_{n})\tau+(p_{n-1}-D_{r} q_{n-1})|\geq
\tau|a_1-D_{r}|,$$ when $N_{\xi}\rightarrow\infty$.
\end{claim}
\begin{proof}When $n=2$ and $a_2\neq1$, or $n=2$ and $a_1>b_{r+1}$, the left hand side simplifies to
\begin{equation}\label{st7}
|(a_1-D_{r})(a_2\tau+1)+\tau|\geq \tau|a_1-D_{r}|.\end{equation}
Now we consider the case $n>2$. If $a_1-b_{r+1}\neq-1$, we have
 {\setlength{\arraycolsep}{0pt}
\begin{eqnarray*}\label{for90}
&& |(p_{n}-D_{r}
q_{n})\tau+(p_{n-1}-D_{r} q_{n-1})|\\
&\geq&|\tau q_{n}(\tfrac{p_{n}}{q_{n}}-D_{r} )|
\\
&\geq&|2\tau(\tfrac{p_{n}}{q_{n}}-D_{r} )|\\
&\geq&\tau|a_1-D_{r}|.
\end{eqnarray*}
}

If $a_1-b_{r+1}=-1$, direct computation shows that $$|D_{r}-
\tfrac{p_{n}}{q_{n}}|=a_1+1+\tfrac{1}{D_{r+1}}-a_1-\frac{1}{a_2+\frac{1}{a_3+\cdots}}\geq
\tfrac{1}{a_{3}+2}$$ and $q_{n}\geq a_{1}a_{3} \geq 2(a_{3}+2)$
when $N_{\xi}\rightarrow\infty$. Hence, we have
 {\setlength{\arraycolsep}{0pt}
\begin{eqnarray*}
&& |(p_{n}-D_{r}
q_{n})\tau+(p_{n-1}-D_{r} q_{n-1})|\\
&\geq&|\tau q_{n}(\tfrac{p_{n}}{q_{n}}-D_{r} )|
\\
&\geq&\frac{\tau q_{n}}{a_{3}+2},
\\
&\geq&\tau|a_1-D_{r}|.
\end{eqnarray*}
}
\end{proof}

We note that $p_{n-1}\geq a_1$. Now if $n=2$ and $a_2\neq1$, or
$n=2$ and $a_1>b_{r+1}$, or $n>2$, combining Claim \ref{5645vv}
and Claim \ref{5645} implies that
 {\setlength{\arraycolsep}{0pt}
\begin{eqnarray*}
&& \frac{2|(p_{n}-D_{r} q_{n})\tau+(p_{n-1}-D_{r}
q_{n-1})||p_{n-1}(\tau'-a_n)-p_{n-2}|}
{D_{r}(\tau\tau'+1)}\\
&\geq&\frac{ 2a_1|a_1-D_{r}|min(1,\tau-1)} {D_{r}(\tau\tau'+1)}
\\
&\geq&\frac{2\vartheta_{u,v}}{(\tau_1\tau_2+1)},
\end{eqnarray*}
}when $N_{\xi}\rightarrow\infty$.

\bigskip

If $n=2$, $a_2=1$ and $a_1<b_{r+1}$, (\ref{st1}) simplifies to
 {\setlength{\arraycolsep}{0pt}
\begin{eqnarray*}
&& \frac{2|(a_1-D_{r})(\tau+1)+\tau||a_{1}(\tau'-1)-1|}
{D_{r}(\tau\tau'+1)}\\
&\geq&\frac{
2(\tau(D_{r}-a_1-1)+D_{r}-a_1)(\tfrac{a_{1}}{\tau}+a_{1}-1)}
{D_{r}(\tau\tau'+1)}
\\
&\geq&\frac{2\vartheta_{u,v}}{(\tau_1\tau_2+1)},
\end{eqnarray*}
}when $N_{\xi}\rightarrow\infty$. This completes the
proof.\end{proof}

Comparing Lemma \ref{3ss376} with Lemma \ref{3sfs376}, we have
$$\lim_{N_{\xi}\rightarrow\infty}c_{\tau_1}\xi <c_{\tau_1}\varphi$$ when $u\geq4$.

\bigskip

We are now in the position to determine the Hurwitz constant of
quadratic Lagrange spectrum for real quadratic number
$[\overline{u,v}]$, $u\geq9$.

\begin{proof}[Proof of Theorem \ref{main3}]
Let $$\xi=[b_0,b_1,\cdots]$$ be an irrational real number not in
$\Theta_{\varphi}\bigcup\Theta_{[\overline{u,v}]}$ and let $d_{r}$
and $D_{r}$ be as before. If we approximate $\xi$  by
\begin{equation}\label{sxscx}\xi'=[b_0,b_1,b_2,\cdots,b_r,k,\tau],\end{equation} we need to verify

\begin{equation}\label{st161} \frac{|(k-D_{r})\tau+1 ||\tau'-k-\tfrac{1}{d_{r}}
|} {D_{r}+\tfrac{1}{d_{r}}}\leq
\frac{(\tau_1-\varphi^2)(1+\tfrac{\tau_2}{\varphi^2})}{\varphi+\tfrac{1}{\varphi}}.\end{equation}
The proof is divided into 4 cases
\bigskip

   \begin{enumerate}
\item  There exist infinitely many $r$ such that $b_r, b_{r+1}>1$.

\bigskip
For such $r$, set $k=[\tau']$. Then the left hand side of
(\ref{st161}) is
$$\frac{|\tau'-D_{r} ||\tau-d_{r} |} {D_{r}d_{r}+1}.$$
which is invariant under the  interchange $(\tau',
D_{r})\leftrightarrow(\tau, d_{r})$. Hence  we can assume without
loss of generality that $d_r\geq D_{r}$. Set $\tau=\tau_2$.

 \begin{claim}\label{5645vvd}
  There exist infinitely many $r$ such that $b_r, b_{r+1}>1$, and
\begin{equation}\label{st164}\frac{|\tau_1-D_{r} |}
{D_{r}d_{r}+1}\leq
\frac{\tau_1-\varphi^2}{\varphi^2(\varphi+\tfrac{1}{\varphi})}.\end{equation}

\end{claim}
\begin{proof}
If $b_r\geq3$, it is easy to check that (\ref{st164})holds.

If  $b_r\geq 3$ does not occur infinitely, then, as $d_r\geq
D_{r}$, either the case $b_{r-1}+1=b_r=b_{r+1}=2\geq b_{r+2}$
occurs infinitely or $b_r=2$ for sufficiently large $r$. We have
in the first case $d_r>2+\tfrac{1}{2}$, $D_{r}\geq2+\tfrac{1}{3}$
in the second case $D_{r}, d_{r}\rightarrow1+\sqrt{2}$. Since
$\tau_1>9$, in both the cases we can verify that  (\ref{st164})
still holds.
\end{proof}

Now choose an $r$ of Claim \ref{5645vvd}. If $\tau_2\leq2d_{r} $,
$\frac{|\tau_2-d_{r} |} {d_{r}}\leq1.$ Since $\tau_2>9$, in this
case we have
$$\frac{|\tau_1-D_{r} ||\tau_2-d_{r} |} {D_{r}d_{r}+1}\leq\frac{|\tau_1-D_{r} |} {D_{r}}\leq\frac{(\tau_1-\varphi^2)(1+\tfrac{\tau_2}{\varphi^2})}{\varphi+\tfrac{1}{\varphi}}.$$

If $\tau_2>2d_{r} $,  by (\ref{st164}), we have,
 {\setlength{\arraycolsep}{0pt}
\begin{eqnarray*}\label{for90}
&& \frac{|\tau_1-D_{r} ||\tau_2-d_{r} |} {D_{r}d_{r}+1}\\
&\leq&\frac{\tau_2(\tau_1-\varphi^2)}{\varphi^2(\varphi+\tfrac{1}{\varphi})}
\\
&\leq&\frac{(\tau_1-\varphi^2)(1+\tfrac{\tau_2}{\varphi^2})}{\varphi+\tfrac{1}{\varphi}}.
\end{eqnarray*}
} This settles the case (1).

\bigskip

If the case (1) is excluded, then the tail of $[b_0,b_1,\cdots]$
has the form
$$b_{c_1},1,\cdots,1,b_{c_2},1,\cdots,1,b_{c_3},1,\cdots,1,b_{c_4},1,\cdots$$
where $b_{c_i}>1$. Set $M=\limsup_{r\rightarrow\infty} b_{c_r}$.
\bigskip
 \item
$M\geq v$.
\bigskip

For any $b_{c_r}\geq v$,  set $k=1$ and $\tau=\tau_2$ , and
replace $r$ with $c_r-1$.

 Then the
left hand side of  (\ref{st161}) is
$$\frac{|1-\tfrac{\tau_2}{D_{c_r}}|(\tau_1-1-\tfrac{1}{d_{c_r-1}})}
{D_{c_r-1}+\tfrac{1}{d_{c_r-1}}}\leq\tau_1-1
\leq\frac{(\tau_1-\varphi^2)(1+\tfrac{\tau_2}{\varphi^2})}{\varphi+\tfrac{1}{\varphi}}.
$$

 \item
$M< u-1$.
\bigskip
 For any $b_{c_r}=M$, set $k=M+1$ and
$\tau=\tau_2$, and replace $r$ with $c_r-1$. Then the left hand
side of (\ref{st161}) is
 {\setlength{\arraycolsep}{0pt}
\begin{eqnarray}\label{for90s}
&&\frac{((\tfrac{1}{1+D_{c_r+1}})\tau_2+1)(\tau_1-M-1-\tfrac{1}{d_{c_r-1}})}
{M+\tfrac{D_{c_r+1}}{1+D_{c_r+1}}+\tfrac{1}{d_{c_r-1}}}\\
&\leq&\frac{(\tfrac{1}{2}\tau_2+1)(\tau_1-\varphi^2)}
{M+\tfrac{1}{2}+\tfrac{1}{2}} \nonumber \\
&\leq&\frac{(\tau_1-\varphi^2)(1+\tfrac{\tau_2}{\varphi^2})}{\varphi+\tfrac{1}{\varphi}}.\nonumber
\end{eqnarray}
}

 \item
$v>M\geq u-1$.
\bigskip

As $M\geq u-1>8$ and  $\tau_1>9$, interchanging $\tau_1$ and
$\tau_2$ in (\ref{for90s}), we get
 {\setlength{\arraycolsep}{0pt}
\begin{eqnarray*}\label{for90}
&&\frac{((\tfrac{1}{1+D_{c_r+1}})\tau_1+1)(\tau_2-M-1-\tfrac{1}{d_{c_r-1}})}
{M+\tfrac{D_{c_r+1}}{1+D_{c_r+1}}+\tfrac{1}{d_{c_r-1}}}\\
&\leq&\frac{(\tfrac{1}{2}\tau_1+1)(\tau_2+\varphi^2)}
{M+\tfrac{1}{2}+\tfrac{1}{2}}\\
&\leq&\frac{(\tau_1-\varphi^2)(1+\tfrac{\tau_2}{\varphi^2})}{\varphi+\tfrac{1}{\varphi}}.
\end{eqnarray*}
}
\end{enumerate}
\end{proof}

\end{document}